\documentclass{amsart}
\usepackage{amsmath,amsthm,amsfonts,amssymb,amscd,systeme}
\usepackage{tikz,tikz-cd}
\usepackage{empheq}
\usepackage{lipsum}
\usepackage{lastpage}
\usepackage{thmtools}
\usepackage{enumerate}
\usepackage[shortlabels]{enumitem}
\usepackage{fancyhdr}
\usepackage{mathrsfs}
\usepackage{xcolor}
\usepackage{graphicx}
\usepackage{mathtools}
\usepackage{listings}
\usepackage{bbm}

\usepackage{hyperref}
\usepackage[makeroom]{cancel}
\usepackage[utf8]{inputenc}

\DeclareMathAlphabet\mathbfcal{OMS}{cmsy}{b}{n}
\DeclareMathOperator{\Sl}{SL}
\DeclareMathOperator{\Gl}{GL}
\DeclareMathOperator{\Span}{Span}

\DeclareMathOperator{\Li}{Li}

\DeclareMathOperator{\lcm}{lcm}
\DeclareMathOperator{\Sym}{sym}

\DeclareMathOperator{\Tr}{Tr}

\setlist[enumerate]{leftmargin=*,widest=0}

\newcommand{\legendre}[2]{\left(\hspace{-1pt}\frac{#1}{#2}\hspace{-1pt}\right)}

\newtheorem{theorem}{Theorem}[section]
\newtheorem{lemma}[theorem]{Lemma}
\newtheorem{Question}{Question}
\theoremstyle{definition}
\newtheorem{definition}{Definition}[section]

\newtheorem{example}[theorem]{Example}

\newtheorem{conjecture}[theorem]{\bf Conjecture}

\theoremstyle{remark}
\newtheorem{remark}{Remark}[section]

\numberwithin{equation}{section}

\begin{document}

\title[Rational spanning sets of level-one cusp forms]{Rational spanning sets of level-one cusp forms from Eisenstein series at prime levels  }

\author{Tianyu Ni}
%%    Address of record for the research reported here
\address{School of Mathematical and Statistical Sciences,
Clemson University,
Clemson, SC 29634-0975,
USA}

\email{tianyuni1994math@gmail.com}

%\author{Author Two}
%\address{Mathematical Research Section, School of Mathematical Sciences,
%Australian National University, Canberra ACT 2601, Australia}

%    General info

%\date{January 1, 2001 and, in revised form, June 22, 2001.}



\begin{abstract}
Let $S_k$ be the space of cusp forms of weight $k$ for the full modular group $\Sl_2(\mathbb{Z})$. In this paper, we will exhibit an explicit rational spanning subset of $S_k$ constructed from Eisenstein series at prime levels. The main ingredient of our proof is to show a result on determining cusp forms by noncentral quadratic-twist
$L$-values.
\end{abstract}

\maketitle

%\tableofcontents
\section{Introduction}
%Let $k\geq2$ be an integer. For $N\geq1$ and a Dirichlet character $\chi$ mod $N$, let $M_k(N,\chi)$ and $S_k(N,\chi)$ be the space of modular forms and cusp forms of weight $k$, level $N$ and nebentypus $\chi$, respectively. We write $M_k(N)$ and $S_k(N)$ if $\chi$ is trivial, and use $M_k$ and $S_k$ to denote $M_k(1)$ and $S_k(1)$. 

For an even integer $k\ge4$, let $M_k$ be the space of modular forms of weight $k$ and level one, and $S_k$ be the corresponding subspace of cusp forms. One interesting fact about modular forms is that the spaces of modular forms can be spanned by forms with rational Fourier coefficients that are of arithmetical interest - for example, the sum of powers of divisors. For even $l\geq4$, the Eisenstein series $E_l$ of weight $l$ is:
\begin{align*}
    E_l(z)=1-\frac{2l}{B_l}\sum_{n=1}^\infty\sigma_{l-1}(n)q^n
\end{align*}
where $B_l$ denotes the $l$-th Bernoulli number, $\sigma_{l-1}(n)=\sum_{d\mid n}d^{l-1}$, and $q=e^{2\pi i z}$ with $z$ in the upper half plane. It is well-known that $M_k$ has the following basis: $E_4^{\alpha}E_6^{\beta}$, where $\alpha,\beta\in\mathbb{Z},\alpha,\beta\geq0$, and $4\alpha+6\beta=k$; see for example \cite[p.~19]{Steinbook}. There is a simple rational spanning subset of $S_k$ :
\begin{align}S_{k}=\mathop{\Span}_{\substack{4\leq l \leq k-4}}(E_{l}\cdot E_{k-l}-E_k)\label{eq:classicalresult},\end{align}
which can be proved by Eichler-Shimura theory \cite{Eichlerperiod, Kohnen1984, Manin1973} on periods of modular forms and the Rankin-Selberg convolution. Finding an explicit basis of $S_k$ consisting of the forms above is rather difficult, and the only result that we are aware of is that of Fukuhara \cite[Theorem 1.2]{Fukuharabasis}.  See also \cite[p.~3]{oddperiods} and \cite[p.~138]{evenperiods} for related general conjectures. Furthermore,  recent works \cite{XueRankin-Cohen, Explicitgenerator2026} give new explicit spanning sets of $S_k$ consisting of Rankin-Cohen brackets of Eisenstein series. One common feature of the aforementioned constructions is that they are built from level-one Eisenstein series of lower weights. It is therefore natural to ask the following question:
\begin{Question}
    Can we construct an explicit rational spanning subset of $S_k$ using only Eisenstein series at higher levels?
\end{Question} 

The main goal of this paper is to give an affirmative answer to the question above. Our method is to take the Rankin-Cohen brackets of Eisenstein series (newforms) at higher levels and trace them down to the level-one space. The idea of this type of construction comes from Kohnen-Zagier \cite[p.~185]{Kohnen-Zagier1981} and Gross-Zagier \cite[p.~271]{GZ-86}, which have recently been applied to various number-theoretic problems; see, for example, \cite{kayath2024subspacesspannedeigenformsnonvanishing}, \cite{twistedperiod}, and \cite{ramanujantauni2026}. Let us first introduce some notation below.
%The Rankin–Cohen bracket of two modular forms \cite{Rankin1956,Cohen'smodularformC_k}, generalizing the product of two modular forms, produces a cusp form. 
\begin{definition}Let $f(z)\in M_{a}(N,\nu_1)$ and $g(z)\in M_{b}(N,\nu_2)$ be modular forms of level $N$, weights $a$ and $b$ and nebentypus $\nu_1$ and $\nu_2$, respectively. For an integer $e\geq1$, the $e$-th Rankin-Cohen bracket is given by
\begin{align*}
    [f(z),g(z)]_e := \sum_{r=0}^e (-1)^r\binom{e+a-1}{e-r}\binom{e+b-1}{r}f(z)^{(r)}g(z)^{(e-r)},
\end{align*}
where $f(z)^{(r)}$ is the $r$-th normalized derivative $f(z)^{(r)}:=\frac{1}{(2\pi i)^r}\frac{d^r f(z)}{dz^r}$ of $f$.  Moreover, $[f,g]_e\in S_{a+b+2e}(N,\nu_1\nu_2)$; see \cite[Theorem 7.1]{Cohen'smodularformC_k}. 
%Note that the definition in Zagier \cite[(73)]{Zagier1976} is related to the above through $F_{e}^{(a,b)}(f(z),g(z))= (-2\pi i)^e e![f(z),g(z)]_e$.
\end{definition}
\begin{definition} Let $p$ be an odd prime and $\chi_p=\legendre{\cdot}{p}$. For an integer $k\geq3$ such that $\chi_p(-1)=(-1)^k$, we have the Eisenstein series (see e.g., \cite[p.~185]{Kohnen-Zagier1981})
\begin{align*}G_{k,\chi_p}(z)&=\sum_{n=0}^\infty \sigma_{k-1,\chi_p}(n)q^n\in M_k(p,\chi_p),\\
 \sigma_{k-1,\chi_p}(n):&=\begin{cases}
        \frac{1}{2}L(1-k,\chi_p)&n=0,\\
        \sum\limits_{d\mid n}\chi_p(d)d^{k-1}&n\geq1.\end{cases}\end{align*}
\end{definition} 
\begin{definition}[{\cite[p.~7]{ramanujantauni2026}}]
For $c\geq1$ and integers $k_1,k_2$ with the same parity such that $3\leq k_1<k_2$ and $\chi_p(-1)=(-1)^{k_1}$, we define 
\begin{align}
\mathcal{F}_{p,k_1,k_2,c}(z)=\Tr^p_1[G_{k_1,\chi_p}(z),G_{k_2,\chi_p}(z)]_c,\label{eq:defofF}
\end{align}
where $\Tr_1^p$  is the trace map 
%(see e.g. \cite[p.271]{GZ-86}) 
\begin{align*} 
    \Tr_1^p: M_{m}(p)\rightarrow M_m(1),\quad g\mapsto \sum_{\gamma\in\Gamma_0(p)\backslash\Gamma_0(1)}g|_m\gamma,
\end{align*}
and for $m\in\mathbb{R}$ and $\gamma=\begin{bsmallmatrix}
    a&b\\c&d
\end{bsmallmatrix}\in \Gl_2^+(\mathbb{R})$  the slash operator \cite[Theorem 7.1]{Cohen'smodularformC_k} is
\begin{align*} 
g(z)|_m\gamma=\det(\gamma)^{m/2} (cz+d)^{-m} g\left(\frac{az+b}{cz+d}\right).
\end{align*}
\end{definition}
\begin{remark}
The condition $k_1<k_2$  is to relate $\mathcal{F}_{p,k_1,k_2,c}$ with noncentral quadratic-twist $L$-values of normalized Hecke eigenforms by the Rankin-Selberg convolution; see Lemma \ref{lem:rankinselberg}. 
The Fourier coefficients of $\mathcal{F}_{p,k_1,k_2,c}$ are given as follows. 
Let $l\geq3$ be an integer, and let $\phi$ and $\psi$ be respectively Dirichlet characters 
modulo $D_1$ and $D_2$ such that  $D=D_1D_2$ and
$\phi(-1)\psi(-1)=(-1)^{l}$. Define \cite[p.~2]{ramanujantauni2026}
    \begin{align*}
   \sigma_{l-1,\phi,\psi}(n):&=\begin{cases}
        -L(1-l,\phi)L(0,\psi)&n=0,\\
\sum\limits_{\substack{d_1,d_2>0\\d_1d_2=n}}\phi(d_1)\psi(d_2)d_1^{l-1}&n\geq1.
    \end{cases}
    \end{align*}For $n\geq1$, let $a_{p,k_1,k_2,c}(n)$ be the $n$-th Fourier coefficient of $\mathcal{F}_{p,k_1,k_2,c}(z)$.   Then  we have (\cite[Proposition 4.3]{ramanujantauni2026})
\begin{align}
        a_{p,k_1,k_2,c}(n)&=\sum_{\substack{a_1,a_2\geq0\\a_1+a_2=n}}\hspace{-5pt}\sigma_{k_1-1,\chi_p,\mathbf{1}}(a_1)\sigma_{k_2-1,\chi_p,\mathbf{1}}(a_2) \cdot c_{a_1,a_2}\nonumber\\&\quad\quad+\chi_p(-1)p^{-c}\sum_{\substack{a_1,a_2\geq0\\a_1+a_2=np}}\hspace{-5pt}\sigma_{k_1-1,\mathbf{1},\chi_p}(a_1)\sigma_{k_2-1,\mathbf{1},\chi_p}(a_2) \cdot c_{a_1,a_2},\label{eq:fourco}\\
        c_{a_1,a_2}:&=\sum_{r=0}^c(-1)^ra_1^ra_2^{c-r}\binom{c+k_1-1}{c-r}\binom{c+k_2-1}{r}.\nonumber
    \end{align}
    In particular, $a_{p,k_1,k_2,c}(n)\in\mathbb{Q}$ for all $n\geq1$. The computation of $a_{p,k_1,k_2,c}(n)$ is rather delicate, see \cite[Section 4]{ramanujantauni2026} for details. 
\end{remark}

We can finally state our main result. For a positive integer $m$ and an integer $a$ such that $\gcd(a,m)=1$, let $\mathscr{P}_{m, a}$ denote the set of primes in the residue class $a\pmod m$.  Note that $|\mathscr{P}_{m, a}|=\infty$ by Dirichlet's theorem on arithmetic progressions.   

\begin{theorem}\label{thm:mainthm}
Let $\mathcal{F}_{p,k_1,k_2,c}(z)$ be defined as in \eqref{eq:defofF}. Then 
\begin{align*}
\mathop{\Span}_{p\in\mathscr{P}_{4,(-1)^{k_1}}}\mathcal{F}_{p,k_1,k_2,c}(z)= S_{k_1+k_2+2c}.
\end{align*}
\end{theorem}
Before going on to sketch the idea of the proof, we make a couple of remarks. First, one of the major differences between our result and the classical one \eqref{eq:classicalresult} is that, here, the indices of the Eisenstein series are fixed, whereas the levels vary. Secondly, a natural question is whether we can find an explicit finite subset $\mathcal{A}$ of $\mathscr{P}_{4,(-1)^{k_1}}$ such that Theorem \ref{thm:mainthm} still holds if we replace $\mathscr{P}_{4,(-1)^{k_1}}$ by $\mathcal{A}$; see also the discussion in Section \ref{sect:twist}. We do not address this question in this paper, but we give some numerical evidence below.
\begin{example}
    If $\dim S_{k_1+k_2+2c}=1$, then $\mathcal{F}_{p,k_1,k_2,c}$ is not identically zero (see \cite[Proposition 3.4]{ramanujantauni2026}), and thus $$\mathbb{C}\mathcal{F}_{p,k_1,k_2,c}=S_{k_1+k_2+2c}.$$
    In particular, this will yield a family of identities that represent Ramanujan's tau function in terms of twisted divisor functions; see \cite{ramanujantauni2026} for details.
\end{example}
\begin{example}
We look at several specific examples when $\dim S_{k_1+k_2+2c}\geq2$. 
Suppose $k_1=4, k_2=6$, and $c=7$. Then $\dim S_{k_1+k_2+2c}=2$. The first two Fourier coefficients of $\mathcal{F}_{p,4,6,7}(z)$ for $p\leq 41$ and $p\equiv 1\pmod 4$ are given in the table \cite{codestrRC}:
\begin{center}
\begin{tabular}{ r  r  r}
 $p$ & $a_{p,4,6,7}(1)$ &  $a_{p,4,6,7}(2)$ \\ 
 \hline
 $5$ & $\frac{154423872}{15625}$  & $-\frac{159476041728}{15625}$\\  
 $13$ & $\frac{155209495337280}{62748517}$ & $-\frac{305452478268057600}{62748517}$\\
 $17$ & $\frac{16698011316480}{1419857}$& $\frac{60160316875776000}{1419857}$\\
 $29$ & $\frac{2769422659862126400}{17249876309}$& $-\frac{4330430663342138910720}{17249876309}$\\
 $37$ & $\frac{76398817748523912000}{94931877133}$ & $-\frac{166849186960968411801600}{94931877133}$ \\
 $41$ & $\frac{260537169091468492800}{194754273881}$ & $\frac{805039626557393707991040}{194754273881}$
\end{tabular}
\end{center}
%\begin{align*}
%\mathcal{F}_{5,4,6,7}(z)&=\frac{154423872}{15625}q-\frac{159476041728}{15625}q^2+O(q^3),
%\\
%\mathcal{F}_{13,4,6,7}(z)&=\frac{155209495337280}{62748517}q-\frac{305452478268057600}{62748517}q^2+O(q^3),\\
%\mathcal{F}_{17,4,6,7}(z)&=\frac{16698011316480}{1419857}q+\frac{60160316875776000}{1419857}q^2+O(q^3),\\
%\mathcal{F}_{29,4,6,7}(z)&=\frac{2769422659862126400}{17249876309}q-\frac{4330430663342138910720}{17249876309}q^2+O(q^3),\\
%\mathcal{F}_{37,4,6,7}(z)&=\frac{76398817748523912000}{94931877133}q-\frac{166849186960968411801600}{94931877133}q^2+O(q^3),\\
%\mathcal{F}_{41,4,6,7}(z)&=\frac{260537169091468492800}{194754273881}q+\frac{805039626557393707991040}{194754273881}q^2+O(q^3).
%\end{align*}
Let $\tilde{a}_{p,k_1,k_2,c}(n)=\frac{a_{p,k_1,k_2,c}(n)}{a_{p,k_1,k_2,c}(1)}$.
Then we can verify that $\tilde{a}_{p_1,4,6,7}(2)\neq\tilde{a}_{p_2,4,6,7}(2)$ for any two distinct primes $p_1$ and $p_2$ in the range above. 
It follows that $\mathcal{F}_{p_1,4,6,7}(z)$ and $\mathcal{F}_{p_2,4,6,7}(z)$  are linearly independent, and thus they span $S_{24}$.

We now suppose $k_1=5, k_2=11$, and $c=4$.  
The values of  $a_{p,5,11,4}(n)$ for $n=1,2$ and $p\leq 31$ and $p\equiv 3\pmod 4$ are given in the table \cite{codestrRC}: 
\begin{center}
\begin{tabular}{ r  r  r}
 $p$ & $a_{p,5,11,4}(1)$ &  $a_{p,5,11,4}(2)$ \\ 
 \hline
 $3$ & $-\frac{16720060}{27}$  & $\frac{2396996000}{9}$\\  
 $7$ & $-\frac{1554085189920}{343}$ & $-\frac{1385938815840000}{343}$\\
 $11$ & $-\frac{706056631667100}{1331}$& $\frac{359581263809856480}{1331}$\\
 $19$ & $-\frac{1096995469227879900}{6859}$& $\frac{422491308698238121440}{6859}$ \\
 $23$ & $-\frac{344508432474842687520}{279841}$ & $-\frac{261274980331859500742400}{279841}$ \\
 $31$ & $-\frac{25593156092907914817600}{923521}$ & $-\frac{22192227722719705312488960}{923521}$
\end{tabular}
\end{center}
Similarly, we have that $\tilde{a}_{p_1,5,11,4}(2)\neq\tilde{a}_{p_2,5,11,4}(2)$ for any two distinct primes $p_1$ and $p_2$ in the range above. 
Thus, $\mathcal{F}_{p_1,5,11,4}$ and $\mathcal{F}_{p_2,5,11,4}$  are linearly independent.

In general, for any $d\geq1$, we can study the linear independence of $\mathcal{F}_{p_i,k_1,k_2,c}$ ($1\leq i\leq d$) by checking the non-singularity of the  matrix formed by the first $d$ Fourier coefficients of each form: 
\begin{align*}
  \begin{pmatrix}
        a_{p_1,k_1,k_2,c}(1) & \cdots &a_{p_1,k_1,k_2,c}(d)\\ \vdots & \ddots &\vdots \\   a_{p_d,k_1,k_2,c}(1) & \cdots &    a_{p_d,k_1,k_2,c}(d)
    \end{pmatrix}.
\end{align*}
Unfortunately, we are unable to show that in the case $\dim S_{k_1+k_2+2c}=2$ for any two distinct $p_1$ and $p_2$ in the residue class $(-1)^{k_1}\pmod 4$, $\mathcal{F}_{p_1,k_1,k_2,c}$ and $\mathcal{F}_{p_2,k_1,k_2,c}$ form a basis of $S_{k_1+k_2+2c}$.  
Nevertheless, the numerical data \cite{codestrRC} suggest the following conjecture. 
\end{example}
\begin{conjecture}Let $d=\dim S_{k_1+k_2+2c}$. Then for any subset $\mathcal{S}\subseteq\mathscr{P}_{4,(-1)^{k_1}}$ with $|\mathcal{S}|=d$, we have 
\begin{align*}
\mathop{\Span}_{p\in\mathcal{S}}\mathcal{F}_{p,k_1,k_2,c}(z)= S_{k_1+k_2+2c}.
\end{align*}
\end{conjecture}

We now sketch the proof. The main idea can be regarded as a generalization of the proof of \eqref{eq:classicalresult}.   It suffices to show that if $g\in S_{k_1+k_2+2c}$ is orthogonal to  $\mathcal{F}_{p,k_1,k_2,c}(z)$ for all $p\in\mathscr{P}_{4,(-1)^{k_1}}$, then $g=0$. We use the Rankin-Selberg convolution and a trick from Kohnen \cite[Theorem 1]{Kohnen2005} to show that there is another cusp form $G_g$ associated to $g$ such that the quadratic-twist $L$-values $L(G_g,\chi_p,k_2+c)=0$ for all $p\in \mathscr{P}_{4,(-1)^{k_1}}$, and satisfies the property: if $G_g=0$, then $g=0$. To proceed with the proof, we need a result on determining a cusp form  by non-central quadratic-twist $L$-values, which is proved in Section \ref{sect:determine}.  Using this result, we obtain $G_g=0$, and thus $g=0$, completing the proof.

The paper is organized as follows. In Section \ref{sect:prelim}, we give all the preparatory results needed in later sections. In Section \ref{sect:determine}, we prove a result on determining cusp forms by noncentral quadratic-twist
$L$-values. We prove Theorem \ref{thm:mainthm} in Section \ref{sect:proof}. Lastly, in Section \ref{sect:twist}, we view our results within the framework of the theory of twisted periods of modular forms and propose questions for future work.

\section{Preliminaries}\label{sect:prelim}
\subsection{Twisted \texorpdfstring{$L$}{L}-values}
Let $f(z)=\sum_{n\geq1}a_f(n)q^n$ be a cusp form in $S_k$, not necessarily a Hecke eigenform. For every primitive Dirichlet character $\chi$ modulo $D$, one can associate  $f$ with a twisted $L$-function:
\begin{align*}
L(f,\chi,s)=\sum_{n=1}^{\infty}a_f(n)\chi(n)n^{-s},
\end{align*}
which converges absolutely for $\Re(s)>\frac{k+1}{2}$ by Deligne's bound $a_f(n)\ll_{f,\epsilon}n^{\frac{k-1}{2}+\epsilon}$. We can view $L(f,\chi,s)$ as the regular Hecke $L$-function $L(f\otimes\chi,s)$ of 
$(f\otimes \chi)(z):=\sum_{n\geq1}a_f(n)\chi(n)q^n$, where
$f\otimes\chi$ is an element of $S_k(D^2,\chi^2)$, see \cite[Proposition 14.19]{Iwaniecbook}.  
The completed $L$-function of $f\otimes\chi$ is defined as (\cite[p.~368]{Iwaniecbook})
\begin{align}
\Lambda(f\otimes\chi,s)=\left(\frac{\sqrt{D^2}}{2\pi}\right)^{s}\Gamma(s)L(f\otimes\chi,s). \label{eq:completeL}
\end{align}
Let $W_{D^2}:S_k(D^2,\chi^2)\rightarrow S_k(D^2,\overline{\chi}^2)$ be the Fricke involution given by:
$$W_{D^2}g(z)=(D^2)^{-k/2}z^{-k}g\left(\frac{-1}{D^2z}\right).$$ Then we have the functional equation (\cite[Theorem 14.7]{Iwaniecbook})
$$\Lambda(f\otimes\chi,s)=i^k\Lambda(W_{D^2}(f\otimes\chi),k-s).$$
Moreover, $L(f\otimes\chi,s)$ is entire.
In particular, if $f$ is a normalized Hecke eigenform in $S_k$, then $f\otimes\chi$ is a newform in $S_k(D^2,\chi^2)$, and the functional equation above becomes (\cite[Proposition 14.20]{Iwaniecbook})
\begin{align}
\Lambda(f\otimes\chi, s)=i^k w \Lambda(f\otimes\overline{\chi},k-s), \label{eq:functionaleqeigenform}
\end{align}
where $w=\frac{\tau(\chi)^2}{D}$ and $\tau(\chi)$ denotes the Gauss sum. 
From now on, $L(f,\chi,s)$ and $L(f\otimes \chi, s) $ are interchangeable for the rest of this paper. 
%and can be analytically continued to the whole complex plane with a certain functional equation. See \cite[pp.~368, 376-377]{Iwaniecbook} for details.
\subsection{Rankin-Selberg convolution}
For two elements $f$ and $g$ of $M_k(N)$ such that $fg$ is a cusp form, the Petersson inner product is given by 
\begin{align*}
    \langle f,g\rangle_{N}=\int_{\Gamma_{0}(N)\backslash\mathbb{H}}f(z)\overline{g(z)}y^kd\mu,
\end{align*}
Here $z=x+iy$ and $d\mu=\frac{dxdy}{y^2}$ is the $\Sl_2(\mathbb{R})$-invariant measure on the upper half plane $\mathbb{H}$.
We can relate $\mathcal{F}_{p,k_1,k_2,c}(z)$ with twisted $L$-value through the Rankin-Selberg convolution. 
\begin{lemma}[{\cite[Proposition 3.3]{ramanujantauni2026}}]\label{lem:rankinselberg} Let $\mathcal{F}_{p,k_1,k_2,c}$ be defined as in \eqref{eq:defofF}.
    If $f$ is a normalized Hecke eigenform in  $S_{k_1+k_2+2c}$, then 
    \begin{align*}
        \langle f, \mathcal{F}_{p,k_1,k_2,c}\rangle_1=&
        %\frac{(-1)^e}{e!}\frac{\Gamma(k+\ell+2e-1)\Gamma(k+e)D^kL(k,\chi)}{(4\pi)^{k+\ell+2e-1}(-2\pi i)^k\overline{\tau(\chi)}L(k,\overline{\chi})}\\&\quad\quad\quad\times
        \tilde{c}\cdot L(f,k_1+k_2+c-1)L(f,\chi_p,k_2+c),\nonumber
    \end{align*}
    where $\tilde{c}$ is some non-zero constant depending only on $k_1,k_2,c$ and $\chi_p$.
\end{lemma}
\begin{remark}\label{remarkeulerprodneq0} Note that  $L(f,k_1+k_2+c-1)$ and $L(f,\chi_p,k_2+c)$ are nonzero because their Euler products converge absolutely when $k_1\leq k_2-2$. \end{remark}
\subsection{Primes in arithmetic progressions}
For positive integers $m$ and $a$ such that $\gcd(m, a)=1$, let $\pi(X;m, a)$ be the counting function
\begin{align*}
    \pi(X;m, a):=\#\{ p\leq X~|~p\in\mathscr{P}_{m,a}\}.
\end{align*}
Recall the Siegel-Walfisz theorem (see e.g., \cite[p.~124]{Iwaniecbook}):
\begin{align}\pi(X;m,a)=\frac{\Li(X)}{\varphi(m)}+o(\Li(X)), \label{eq:SiegelWalfisz}
\end{align}
where $\Li$ denotes the logarithmic integral, $\Li(X)\sim \frac{X}{\log X}$, and $\frac{o(\Li(X))}{\Li(X)}\rightarrow 0$ as $X\rightarrow\infty$.
The following fact is  likely well-known to experts. But we include a proof here for completeness.

\begin{lemma}\label{lem:arithmeticprogression}
For a fixed integer $t\geq1$, we write $t=u^2d$ with $d$ square-free, and let $\Delta_t$ be the fundamental discriminant associated to $t$:
$\Delta_t=d$ if $d\equiv1\pmod 4$ and $\Delta_t=4d$ if $d\equiv 2,3 \pmod 4$. We write $\chi_{\Delta_t}=\legendre{\Delta_t}{\cdot}$. Then we have 
\begin{align}
\lim_{X\rightarrow \infty}\frac{1}{\pi(X;m,a)}\sum_{\substack{p\leq X\\ p\equiv a\pmod m}}\legendre{t}{p}=\begin{cases} \chi_{\Delta_t}(a) & \Delta_t\mid m, \\ 0 & \Delta_t\nmid m.\end{cases} \label{eq:limit}
\end{align}
\end{lemma}
\begin{proof}
    Let $Q=\lcm(\Delta_t,m)$. Define 
$$\mathcal{Q}_a:=\{y\in(\mathbb{Z}/Q\mathbb{Z})^{\times}~:~y\equiv a \pmod m\}.$$
Partitioning the residue class $a\pmod m$ into classes modulo $Q$, we get
\begin{align}\pi(X;m,a)=\sum_{b\in\mathcal{Q}_a}\pi(X;Q,b)+ O_{t,m}(1),\label{eq:rewrite2}
\end{align}
where $O_{t,m}(1)$ is bounded independent of $X$.
Note also that $\legendre{t}{p}=\chi_{\Delta_t}(p)$ for all but the finitely many primes dividing $2t$. Then
\begin{align}
\sum_{\substack{p\leq  X\\ p\equiv a\pmod m}}\legendre{t}{p}&=\sum_{\substack{p\leq  X\\ p\equiv a\pmod m}}\chi_{\Delta_t}(p)+O_{t,m}(1)\nonumber\\&=\sum_{b\in\mathcal{Q}_a}\chi_{\Delta_t}(b)\pi(X;Q,b) + O_{t,m}(1). \label{eq:rewrite1}
\end{align}
From \eqref{eq:rewrite2}, \eqref{eq:rewrite1}, and \eqref{eq:SiegelWalfisz}, 
the limit \eqref{eq:limit} becomes
\begin{align}
\lim_{X\rightarrow \infty}\frac{\sum_{b\in\mathcal{Q}_a}\chi_{\Delta_t}(b)\left(\frac{1}{\varphi(Q)}+\frac{o(\Li (X))}{\Li(X)}\right)+\frac{O_{t,m}(1)}{\Li(X)}}{\sum_{b\in\mathcal{Q}_a}\left(\frac{1}{\varphi(Q)}+\frac{o(\Li (X)}{\Li(X)})\right)+\frac{O_{t,m}(1)}{\Li(X)}}
=\frac{1}{|\mathcal{Q}_a|}\sum_{b\in\mathcal{Q}_a}\chi_{\Delta_t}(b).\label{eq:midstep3}
\end{align}
Now let $H$ be the kernel of the natural projection $\pi:(\mathbb{Z}/Q\mathbb{Z})^{\times}\rightarrow(\mathbb{Z}/m\mathbb{Z})^{\times}$. Choose any $b_0\in\mathcal{Q}_a$. Then we have
\begin{align*}
    \frac{1}{|\mathcal{Q}_a|}\sum_{b\in\mathcal{Q}_a}\chi_{\Delta_t}(b)=\chi_{\Delta_t}(b_0)\frac{1}{|H|}\sum_{b\in H}\chi_{\Delta_t}(b).
\end{align*}
 By the orthogonality relations, the above character sum is zero unless $\chi_{\Delta_t}$ is trivial on $H$, equivalently, unless $\chi_{\Delta_t}$ factors through $(\mathbb{Z}/m\mathbb{Z})^{\times}$, which occurs exactly when $\Delta_t\mid m$. Note also that if $\chi_{\Delta_t}$ is trivial on $H$, then $\chi_{\Delta_t}(b_0)=\chi_{\Delta_t}(a)$. This gives the desired result.
\end{proof}
\begin{remark}\label{remark:square}
    Note that when $m=4$, $\Delta_t\mid 4$ if and only if $t$ is a square, and thus the limit \eqref{eq:limit} is $1$ if $t$ is a square, and $0$ otherwise.
\end{remark}
\subsection{Symmetric square \texorpdfstring{$L$}{L}-functions}
Let $f(z)=\sum_{n\geq1}a_f(n)q^n$ be a normalized Hecke eigenform in $S_k$. For each prime $p$, we denote by $\alpha_p$ and $\beta_p$ the roots of 
$$X^2-a_f(p)X+p^{k-1}=0.$$
The symmetric square $L$-function attached to $f$ for $\Re(s)\gg0$  is defined by
\begin{align*}
L(\Sym^2 f, s)=\prod_p(1-\alpha_p^2 p^{-s})^{-1}(1-\alpha_p\beta_p p^{-s})^{-1}(1-\beta_p^2p^{-s})^{-1},
\end{align*}
which has an analytic continuation to the whole complex plane and satisfies a functional equation \cite[p.~109]{Zagier1976}. If $\Re(s)>k$, then the series $\sum_{n\geq1}a_f(n^2)n^{-s}$ converges absolutely by Deligne's theorem, and we have \cite[p.~376]{Coh17}
\begin{align}
L(\Sym^2 f,s)=\zeta(2s+2-2k)\sum_{n=1}^{\infty}\frac{a_f(n^2)}{n^s}.\label{eq:symsquareandseries}
\end{align}
In particular, $L(\Sym^2 f,s)\neq 0$ for $\Re(s)>k$ since the Euler product converges absolutely. 
We need the following simple fact in the next section. 
\begin{lemma}\label{lem:symsquare}
Let $f(z)=\sum_{n\geq1}a_f(n)q^n$ be a normalized Hecke eigenform in $S_k$, and let $s$ be a complex number such that $\Re(s)>\frac{k+1}{2}$. Then for any fixed square-free integer $r\geq1$, we have 
\begin{align}
\sum_{m=1}^{\infty}\frac{a_f(rm^2)}{(rm^2)^s}=A_{s,r,k}L(\Sym^2f,2s)a_f(r),
\end{align}
where $A_{s,r,k}$ is a nonzero constant depending only on $s,r$ and $k$, given by 
\begin{align*}
  A_{s,r,k}=r^{-s}\zeta(4s+2-2k)^{-1}\prod_{p\mid r}(1+p^{k-1-2s})^{-1}.
\end{align*}
\end{lemma}
\begin{proof}
  Note that $f$ is a normalized Hecke eigenform. Then the Hecke relation gives 
\begin{align*}
a_f(r)a_f(m^2)=\sum_{d\mid(r,m^2)}d^{k-1}a_f\left(\frac{rm^2}{d^2}\right).
\end{align*} 
Since $r$ is square-free, $d\mid(r,m^2)$ if and only if $d\mid r$ and $d\mid m$. As $s$ lies within the region of absolute convergence for all series involved, we  use the above Hecke relation to get 
\begin{align}
a_f(r)\sum_{m=1}^{\infty}\frac{a_f(m^2)}{m^{2s}}&=\sum_{m=1}^{\infty}\sum_{d\mid r,m}d^{k-1}a_f\left(\frac{rm^2}{d^2}\right)m^{-2s}\nonumber\\&=\sum_{d\mid r}d^{k-1}\sum_{\substack{m\geq1\\d\mid m}}a_f\left(\frac{rm^2}{d^2}\right)m^{-2s}\nonumber\\&=\sum_{d\mid r}d^{k-1-2s}\sum_{n=1}^{\infty}\frac{a_f(rn^2)}{n^{2s}}\nonumber\\&=\prod_{p\mid r}(1+p^{k-1-2s})\sum_{n=1}^{\infty}\frac{a_f(rn^2)}{n^{2s}},\label{eq:producthecke}
\end{align}
where we used the fact that $\sum_{d\mid r}d^{k-1-2s}=\prod_{p\mid r}(1+p^{k-1-2s})$ when $r$ is square-free in the last equality. Now applying \eqref{eq:symsquareandseries} to \eqref{eq:producthecke} gives the result.
\end{proof}

\section{Determination of cusp forms by noncentral quadratic-twist \texorpdfstring{$L$}{L}-values}\label{sect:determine}

There are many works on determining cusp forms by twists of $L$-values. One of the most influential results is the work by Luo and Ramakrishnan \cite{luo1997}, which showed that a certain set of central twisted $L$-values, up to a constant, is sufficient to determine a normalized holomorphic newform. 
Our result differs from previous work in two respects. Firstly, we consider the determination of general cusp forms, not necessarily normalized newforms. Secondly, we consider a set of noncentral quadratic-twist $L$-values rather than central ones, which is technically much simpler. 

We need the following fact on determining cusp forms by square-free Fourier coefficients due to Anamby and Das \cite{Das2019}.
\begin{lemma}\label{lem:squarefree}
Let $f(z)=\sum_{n\geq1}a_f(n)q^n\in S_k$ be such that $a_f(r)=0$ for all square-free integers $r$. Then $f=0$. 
\end{lemma}
\begin{proof}
This is a special case of \cite[Theorem 6]{Das2019} for $N=1$.
\end{proof}

\begin{theorem}\label{thm:twistvanishing}
Fix $b\in\{1,3\}$. Let $f$ be a cusp form in $S_k$, not necessarily a Hecke eigenform, and let $s$ be a complex number such that $\Re(s)>\frac{k+1}{2}$. If $L(f,\chi_p,s)=0$ for all primes in the residue 
class $b\pmod 4$, then $f=0$.
\end{theorem}
\begin{proof}
 Let $r\geq1$ be an arbitrary square-free integer. Then we have by assumption
\begin{align*}
\chi_p(r)L(f,\chi_p,s)=0,
\end{align*}
for all $p\in\mathscr{P}_{4,b}$.
Taking the average, we get 
\begin{align*}
\frac{1}{\pi(X;4,b)}\sum_{\substack{p\leq X\\ p\equiv b \pmod 4}}\chi_p(r)L(f,\chi_p,s)=0.
\end{align*}
Let $\sigma=\Re(s)$. Since $\sigma>\frac{k+1}{2}$, we have $\sum_{n\geq1}|a_f(n)|n^{-\sigma}<\infty$
by Deligne's bound. Note also that for every $X$ and $n$, 
\begin{align*} 
\left|\frac{1}{\pi(X;4,b)}\sum_{\substack{p\leq X\\ p\equiv b \pmod 4}}\legendre{rn}{p}\right|\leq 1. 
\end{align*}
Hence, we can interchange the limit and the sum, obtaining 
\begin{align}
0&=\lim_{X\rightarrow \infty}\frac{1}{\pi(X;4,b)}\sum_{\substack{p\leq X\\ p\equiv b \pmod 4}}\chi_p(r)L(f,\chi_p,s)\nonumber\\&=\sum_{n=1}^{\infty}\frac{a_f(n)}{n^s}\lim_{X\rightarrow \infty}\frac{1}{\pi(X;4,b)}\sum_{\substack{p\leq X\\ p\equiv b \pmod 4}}\legendre{rn}{p}.\label{eq:limitinside}
\end{align}
Since $r$ is square-free, $rn$ is a square if and only if $n=rm^2$. Applying Lemma \ref{lem:arithmeticprogression} to \eqref{eq:limitinside}, we get 
\begin{align}
\sum_{m=1}^{\infty}\frac{a_f(rm^2)}{(rm^2)^s}=0.\label{eq:vanish}
\end{align}
 Suppose $f=\sum_{j=1}^dc_j f_j$, where $f_j$'s ($1\leq j\leq d$) are the normalized Hecke eigenforms of $S_k$, and each $f_j$ has Fourier coefficients $a_{f_j}(n)$.   Then \eqref{eq:vanish} gives 
\begin{align*}
    \sum_{j=1}^dc_j\sum_{m=1}^{\infty}\frac{a_{f_j}(rm^2)}{(rm^2)^s}=0.
\end{align*}
By Lemma \ref{lem:symsquare} and the fact that $A_{s,r,k}\neq0$, 
we obtain that 
\begin{align}
    \sum_{j=1}^dc_jL(\Sym^2f_j,2s)a_{f_j}(r)=0. \label{eq:vanishco}
\end{align}
We now define another cusp form associated to $f$ by
\begin{align*}
    F_f(z)=\sum_{j=1}^d c_j L(\Sym^2f_j,2s)f_j.
\end{align*}
Then the $r$-th Fourier coefficient of $F_f$ vanishes by \eqref{eq:vanishco}. Because this holds for any square-free integer $r$, we must have $F_f=0$ by Lemma \ref{lem:squarefree}. Since $L(\Sym^2 f_j,2s)\neq0$, we have $c_j=0$, and thus $f=0$. This completes the proof.
\end{proof}

\section{Proof of Theorem \ref{thm:mainthm}}\label{sect:proof}
%Finally, we are ready to prove Theorem \ref{thm:mainthm}.
\begin{proof}[Proof of Theorem \ref{thm:mainthm}]
  It suffices to show that if $g\in S_{k_1+k_2+2c}$ is orthogonal to  $\mathcal{F}_{p,k_1,k_2,c}(z)$ for all $p\in\mathscr{P}_{4,(-1)^{k_1}}$, then $g=0$.  Write $g=\sum_{j=1}^dc_jf_j$, where $f_j$'s are normalized Hecke eigenforms in $S_{k_1+k_2+2c}$. By Lemma \ref{lem:rankinselberg}, we have
\begin{align*}
        \langle f_j, \mathcal{F}_{p,k_1,k_2,c}\rangle_1=&
        %\frac{(-1)^e}{e!}\frac{\Gamma(k+\ell+2e-1)\Gamma(k+e)D^kL(k,\chi)}{(4\pi)^{k+\ell+2e-1}(-2\pi i)^k\overline{\tau(\chi)}L(k,\overline{\chi})}\\&\quad\quad\quad\times
        \tilde{c}\cdot L(f_j,k_1+k_2+c-1)L(f_j,\chi_p,k_2+c),
    \end{align*}
where $\tilde{c}$ is a nonzero constant independent of $f_j$. Thus, the condition on orthogonality implies that
\begin{align}
\sum_{j=1}^dc_jL(f_j,k_1+k_2+c-1)L(f_j,\chi_p,k_2+c)=0,\label{eq:mid}
\end{align}
for all $p\in\mathscr{P}_{4,(-1)^{k_1}}$. Following the idea of \cite[Theorem 1]{Kohnen2005}, we define another cusp form associated to $g$ by:
\begin{align*}
G_g(z)=\sum_{j=1}^d c_j L(f_j,k_1+k_2+c-1)f_j(z).
\end{align*}
Now, \eqref{eq:mid} implies that 
$L(G_g,\chi_p,k_2+c)=0$
for all $p\in\mathscr{P}_{4,(-1)^{k_1}}$. Note that $k_2+c>\frac{k_1+k_2+2c+1}{2}$. Then by Theorem \ref{thm:twistvanishing}, we have $G_g=0$. Since each $L(f_j,k_1+k_2+c-1)\neq 0$, we must have $c_j=0$, and thus $g=0$. 
 Thus, the proof of Theorem \ref{thm:mainthm} is complete.
\end{proof}
\section{Discussion and future work}\label{sect:twist}
%In this section, we will view our results in a general framework of the theory of twisted periods. We also raise several questions for future work.

 Let $\chi$ be a primitive Dirichlet character modulo $D$, and let $f$ be a cusp form in $S_k$, not necessarily a Hecke eigenform. For $0\leq t\leq k-2$,  the $t$-th twisted period for $f\in S_k$ is defined as \cite[p.~978]{twistedLvaluesFukuhara}
    \begin{align}
        r_{t,\chi}(f):=\int_0^{i\infty} (f\otimes{\chi})(z) z^t dz=\frac{t!}{(-2\pi i)^{t+1}}L(f\otimes\chi,t+1).\label{eq:periodandLvalue}
    \end{align}
%where $(f\otimes\chi)(z):=\sum_{n\geq1}a_f(n)\chi(n)q^n$ is an element in $S_{k}(D^2,\chi^2)$, see \cite[Proposition 14.19]{Iwaniecbook}.
In particular, when $D=1$, we obtain the regular periods $r_t$ ($0\leq t\leq k-2)$.  The result of Eichler-Shimura \cite{Eichlerperiod,Kohnen1984,Manin1973} asserts that the even periods $r_0,r_2,\dots,r_{k-2}$ span the vector space $S_k^{\ast}$. Similarly, the odd periods $r_1,r_3,\dots,r_{k-3}$ span $S_k^{\ast}$ as well. Moreover, these periods are not linearly independent because both sets of even and odd periods are of size $> \dim S_{k}$. In fact, they are subject to many linear dependence relations called the Eichler-Shimura relations \cite{Manin1973}.
%;  see \eqref{eq:ES1}, \eqref{eq:ES2}, and \eqref{eq:ES3} below. 
%\begin{gather}
%    \quad r_t+(-1)^{t}r_{k-2-t}=0,\label{eq:ES1}\tag{$1_t$}
%    \\\quad(-1)^{t}r_t+\sum_{\substack{0\leq \ell\leq t\\\ell\equiv0\pmod2}}\binom{t}{\ell}r_{k-2-t+\ell}+\sum_{\substack{0\leq \ell\leq k-2-t\\\ell\equiv t\pmod2}}\binom{k-2-t}{\ell}r_{\ell}=0,\label{eq:ES2}\tag{$2_t$}
%    \\\quad\sum_{\substack{1\leq \ell\leq t\\\ell\equiv1\pmod2}}\binom{t}{\ell}r_{k-2-t+\ell}+\sum_{\substack{0\leq \ell\leq k-2-t\\\ell\not\equiv t\pmod2}}\binom{k-2-t}{\ell}r_{\ell}=0.\label{eq:ES3}\tag{$3_t$}
%\end{gather}

This motivates the development of an Eichler-Shimura theory for twisted periods. For example, in \cite[p.~2]{twistedperiod}, the authors asked several questions about these periods: the first question about these periods is whether they span the whole space of $S_k^{\ast}$. The second question is whether we can find Eichler-Shimura-type linear relations  among them. 
Thirdly, we would like to know which periods are linearly independent. 

Unfortunately, the very first question was not solved in \cite{twistedperiod}; instead, the authors provided evidence for the linear independence of periods in two scenarios \cite[Theorems 1.1-1.4]{twistedperiod}: 
\begin{itemize}
    \item[$(C_t):$]\label{item:caseA} fixed character $\chi$ with different indices $t$; 
    \item[$(C_{\chi}):$]\label{item:caseB} fixed index $t$ with different characters $\chi$ with the same conductor $D$.
\end{itemize}
Now we can answer the aforementioned first question in the setting where $t$ is fixed while $\chi_p$ takes different quadratic primitive characters modulo $p$.  
\begin{theorem}\label{thm:twperiod}
    Fix $b\in\{1,3\}$. Let $k$ be an even integer and $t$ be an integer such that $\frac{k}{2}\leq t \leq k-2$ or $0\leq t \leq \frac{k}{2}-2$. Then 
    \begin{align}\bigcap_{p\in\mathscr{P}_{4,b}}\ker (r_{t,\chi_p}) = 0.\label{eq:ker}\end{align}
    In particular, the periods $r_{t,\chi_p}$ ($p\in \mathscr{P}_{4,b}$) span $S_k^{\ast}$.
\end{theorem}
\begin{proof}
The first statement for $\frac{k}{2}\leq t\leq k-2$ follows immediately from Theorem \ref{thm:twistvanishing}.  Note that $\chi_p=\overline{\chi}_p$ and $i^k\frac{\tau(\chi_p)^2}{p}$ is independent of the Hecke eigenforms in $S_k$. Then the functional equation \eqref{eq:functionaleqeigenform} extends by linearity to every $f\in S_k$, and thus \eqref{eq:ker} holds for $0\leq t\leq \frac{k}{2}-2$ as well.  The second statement follows from the fact that $\dim S_k<\infty$.
\end{proof}
We conclude the paper by raising two questions. First,  note that $|\mathcal{P}_{4, \pm1}|=\infty$, so a natural question is whether we can find an explicit finite spanning subset. In the spirit of Lemma \ref{lem:arithmeticprogression} and the proof of  Theorem \ref{thm:twistvanishing}, we ask the following question.
\begin{Question}Fix $b\in\{1,3\}$ and an integer $t$ such that $0\leq t\leq \frac{k}{2}-2$ or $\frac{k}{2}\leq t\leq k-2$.
Can we  find an effective number $X=X(k,t)$ such that
$$\bigcap_{p\leq X,~p\in\mathscr{P}_{4,b}}\ker (r_{t,\chi_p}) = 0\quad?$$
\end{Question}
Secondly, we ask the following question about linear independence, which differs from the settings \hyperref[item:caseA]{$(C_t)$} and \hyperref[item:caseB]{$(C_{\chi})$}.  
\begin{Question}
    Fix $b\in\{1,3\}$, an integer $t\geq 0$, and 
    primes $p_1<p_2<\cdots< p_d$ in the residue class $b\pmod 4$. Can we find an effective number 
    $Y=Y(t,p_1,\dots,p_d)$
    such that $r_{t,\chi_{p_1}},\dots,r_{t,\chi_{p_d}}\in S_k^{\ast}$ are linearly independent for every even integer $k > \max(Y, 2t+2)$? 
\end{Question}

\section*{Acknowledgements}
The author would like to thank Hui Xue for the discussion. 

\bibliographystyle{plain}
  
  \vspace{.2in}

\providecommand{\bysame}{\leavevmode\hbox
to3em{\hrulefill}\thinspace}

\bibliography{ref}

\end{document}